\documentclass[12pt, a4paper]{amsproc}
\usepackage{amsmath,amssymb,anysize,times,float,enumerate}
\usepackage{listings}

\usepackage{hyperref} 
\hypersetup{citecolor=red, linkcolor=blue, colorlinks=true}

\renewcommand\le{\leqslant}
\renewcommand\ge{\geqslant}

\newtheorem{thm}{Theorem}[section]
\newtheorem{lem}[thm]{Lemma}

\theoremstyle{definition}

\theoremstyle{remark}

\newcommand\Wr{{\rm \,wr\, }}

\DeclareMathOperator\PGL{{\rm PGL}}
\DeclareMathOperator\AGL{{\rm AGL}}

\DeclareMathOperator\Out{{\rm Out}}
\DeclareMathOperator\Sz{{\rm Sz}}

\title[Covering number of symmetric groups]{On the covering number of symmetric groups of even degree}

\author{Eric Swartz}
\address{ Department of Mathematics, College of William and Mary, P.O. Box 8795, Williamsburg, VA 23187-8795}
\email{easwartz@wm.edu}



\begin{document}

\begin{abstract}
If a group $G$ is the union of proper subgroups $H_1, \dots, H_k$, we say that the collection $\{H_1, \dots H_k \}$ is a cover of $G$, and the size of a minimal cover (supposing one exists) is the covering number of $G$, denoted $\sigma(G)$.  Mar\'{o}ti showed that $\sigma(S_n) = 2^{n-1}$ for $n$ odd and sufficiently large, and he also gave asymptotic bounds for $n$ even.  In this paper, we determine the exact value of $\sigma(S_n)$ when $n$ is divisible by $6$. 
\end{abstract}

\maketitle

\section{Introduction}

Let $G$ be a group and ${\mathcal A} = \{A_i : 1\leqslant i\leqslant n\}$ be a collection of proper subgroups of $G$.  If $G = {\bigcup\limits_{i=1}^n}A_i$ (as a set theoretic union),
then ${\mathcal A}$ is called a \textit{cover} of $G$.  A cover of size $n$ is said to be \textit{minimal} if no cover of $G$ has fewer than $n$ members.  The size of a minimal covering of $G$, supposing one exists, is called the \textit{covering number} and is denoted by $\sigma(G)$.

The concept of a cover is only well-defined if $G$ is not a cyclic group. Indeed, if $G$ is a cyclic group, then no generator of $G$ is contained in a proper subgroup, and so $G$ has no cover.  On the other hand, if $G$ is not cyclic, then one could take all cyclic subgroups as a cover.  Moreover, when considering the covering number of a group, it is obvious that the subgroups used in a cover can be restricted to maximal subgroups.

Note that we can consider covers of either finite or infinite groups.  Indeed, B.H. Neumann \cite{N} showed that a group is the union of finitely many proper subgroups if and only if it has a finite noncyclic homomorphic image.  In this paper we will restrict ourselves to finite groups.

The covering number $\sigma(G)$ of a finite group $G$ provides an upper bound for $\omega(G)$, which is defined to be the largest integer $m$ such that there exists a subset $S$ of $G$ of size $m$ with the property that any two distinct elements of $S$ generate $G$.  There has been a great interest in this topic in recent years (see \cite{B, BEGHM1, BEGHM2, HM}), especially with regards to the application of $\sigma(G)$ as an upper bound for $\omega(G)$.  For a survey regarding the covering number and related problems, see \cite{Se}. 

In \cite{C}, Cohn conjectures that the covering number of any (noncyclic) solvable group has the form $p^{\alpha}+1$, where $p$ is a prime and $\alpha$ is a positive integer.  In \cite{T}, Tomkinson confirms this conjecture, showing that the covering number of any (noncyclic) solvable group has the form $|H/K| + 1$, where $H/K$ is the smallest chief factor of $G$ having more than one complement in $G$.

Furthermore, Tomkinson suggests that it might be of interest to determine the covering number of simple groups.  Along these lines, the covering number for 2-dimensional linear groups was determined by Bryce, Fedri, and Serena in \cite{BFS}, and the covering number for the Suzuki groups $\Sz(q)$ was determined by Lucido in \cite{Lu}.  Holmes applied innovative combinatorial and computational techniques using GAP \cite{Ga} in \cite{H} to calculate the covering number of many sporadic simple groups.  

Naturally, there has been great interest in symmetric and alternating groups.   Mar\'{o}ti made great progress on both in \cite{M}.  For alternating groups, Mar\'{o}ti showed that $\sigma(A_n) \geqslant 2^{n-2}$ for $n \neq 7,9$ with equality if and only if $n \equiv 2 \pmod 4$ and further proved that $\sigma(A_7) \leqslant 31$ and $\sigma(A_9) \geqslant 80.$  Small values of $n$ have been resolved elsewhere.  Cohn \cite{C} showed that $\sigma(A_5) = 10$; Kappe and Redden \cite{KR} showed that $\sigma(A_7) = 31$, $\sigma(A_8) = 71$, and $127 \leqslant \sigma(A_9) \leqslant 157$; and recently Epstein, Magliveras, and Nikolova-Popova \cite{EMN} showed that $\sigma(A_9) = 157$ and $\sigma(A_{11}) = 2751$.  

For symmetric groups, Mar\'{o}ti showed for $n$ odd that $\sigma(S_n) = 2^{n-1}$ unless $n = 9$ and showed for $n$ even that $\sigma(S_n) \sim \frac{1}{2} {\binom{n}{n/2}}.$  We note that $\sigma(S_4) = 4$ by \cite{T} and $\sigma(S_6) = 13$ by \cite{AAS}.  Kappe, Nikolova-Popova, and the author showed in \cite{KNS} that $\sigma(S_8) = 64$, $\sigma(S_9) = 256$ (confirming that $\sigma(S_n) = 2^{n-1}$ for all odd $n$), $\sigma(S_{10}) = 221,$ and $\sigma(S_{12}) = 761$, establishing that the upper bound of $761$ for $\sigma(S_{12})$ Mar\'{o}ti gave in \cite{M} was in fact the exact value.    

It is obvious that computational methods can only be taken so far with symmetric groups of even degree, and the goal of this paper is analyze these groups in the same spirit as \cite{M}.  We will prove the following theorem:

\begin{thm}
\label{thm:6n}
Let $n \equiv 0 \pmod {6}$, $n \geqslant 24$. If $\sigma(S_n)$ denotes the subgroup covering number of $S_n$, then $\sigma(S_n) = \frac{1}{2} {\binom{n}{n/2}} + \sum\limits_{i=0}^{n/3 - 1} {\binom{n}{i}}.$  Moreover, $\sigma(S_{18}) = 36772 = \frac{1}{2} {\binom{18}{9}} + \sum\limits_{0 \leqslant i \leqslant 5, i \neq 2} {\binom{18}{i}}.$  In each of these cases, the minimal cover using only maximal subgroups is unique.  
\end{thm}

The following notation will be used throughout the paper.  Given an element $g \in S_n$, we say that the permutation $g$ has cycle structure $(n_1,...,n_k)$ with $n_1 \le n_2 \le \dots \le n_k$ if $g$, when written as the product of disjoint cycles, contains cycles of length $n_i$ for $1 \leqslant i \leqslant k$, where $\sum\limits_{i=1}^k n_i = n$.  For instance, the permutation $(1\text{ }2)(3\text{ }4)(5\text{ }6\text{ }7\text{ }8\text{ }9) \in S_9$ has cycle structure $(2,2,5)$, whereas the permutation $(1\text{ }2)(3\text{ }4)(5\text{ }6\text{ }7\text{ }8\text{ }9) \in S_{10}$ has cycle structure $(1,2,2,5)$. 

This paper is organized as follows: in Section \ref{sect:sub}, we provide details about maximal subgroups of symmetric groups; in Section \ref{sect:lemma}, we prove a lemma that provides a sufficient condition for a cover consisting of entire conjugacy classes of maximal subgroups to be minimal; in Section \ref{sect:small}, we apply this lemma to the groups $S_{18}$ and $S_{24}$ to establish their covering numbers; and, finally, in Section \ref{sect:large}, we apply the lemma to establish the covering number of $S_n$, where $n \ge 30$ and $n \equiv 0 \pmod 6$. 

\section{Subgroups of symmetric groups}
\label{sect:sub}

The maximal subgroups of the symmetric group $S_n$ are characterized by the O'Nan-Scott Theorem, which may be stated as follows:

\begin{thm}[{\cite{Sc}}]
\label{thm:onan}
Let $H$ be a maximal subgroup of of $S_n$.  Then $H$ is isomorphic to one of the following:
\begin{itemize}
 \item[(i)] $S_k \times S_\ell$, where $k + \ell = n$;
 \item[(ii)] $S_k \Wr S_\ell$, where $k\ell = n$;
 \item[(iii)] $S_k \Wr S_\ell$, where $k^\ell = n$ and $k > 2$;
 \item[(iv)] $\AGL(d,p)$, where $p^d = n$;
 \item[(v)] $T^k.(\Out(T) \times S_k)$, where $T$ is a nonabelian simple group and $|T|^{k-1} = n$;
 \item[(vi)] an almost simple group.
\end{itemize}
\end{thm}

For the purposes of this paper, with the exception of singling out the alternating group $A_n$, there is no need to distinguish between subgroups that fall under (iii) -- (vi) of Theorem \ref{thm:onan}.  Identifying $S_n$ with its natural action on $\{1, \dots, n\}$, we will instead divide the maximal subgroups of $S_n$ into the following four classes:  

\begin{itemize}
 \item[(1)] The \textit{alternating group} $A_n$.
 \item[(2)] \textit{Intransitive groups}, i.e., those groups isomorphic to $S_k \times S_\ell$, $k + \ell = n$, which stabilize a decomposition of the set $\{1,\dots, n\}$ into one set of size $k$ and one set of size $\ell$.
 \item[(3)] \textit{Imprimitive groups}, i.e., those groups isomorphic to $S_k \Wr S_\ell$, where $k\ell = n$, which stabilize a decomposition of the set $\{1,\dots, n\}$ into $\ell$ sets of size $k$.  Note that, unlike the intransitive groups, imprimitive groups are transitive on the set $\{1,...,n\}$.
 \item[(4)] \textit{Primitive groups}, i.e., those groups that act primitively on $\{1, \dots, n\}$ and are not the alternating group $A_n$. These are the groups that are not $A_n$ and fall under (iii)--(vi) of Theorem \ref{thm:onan}.
\end{itemize}

For large values of $n$, the primitive groups (that are not $A_n$) have orders that are very small compared to the orders of the maximal subgroups in classes (1) -- (3):  

\begin{lem}[{\cite[Corollary 1.2]{M2}}]
\label{lem:primbound} 
If $G$ is a primitive subgroup of $S_n$ that is not the alternating group $A_n$ and $n > 24$, then $|G| < 2^n$.
\end{lem}

\section{A sufficient condition for a cover to be minimal}
\label{sect:lemma}

Let $G$ be a finite group whose conjugacy classes maximal subgroups are indexed by a set $I_G$.  For $i \in I_G$, let $\mathcal{M}_i$ denote a conjugacy class of maximal subgroups.  Let $\Pi$ be a union of conjugacy classes of elements of $G$.  Assume that, for some $I \subseteq I_G$, $\mathcal{C} = \bigcup\limits_{i \in I} \mathcal{M}_i$ is a cover of $\Pi$, and assume that the elements of $\Pi$ are partitioned among the subgroups in $\mathcal{C}$.  Denote by $\Pi_i$ the set of elements covered the conjugacy class $\mathcal{M}_i$.  Note that, since $\Pi$ is a union of conjugacy classes of elements of $G$, for each conjugacy class $\mathcal{M}$ of maximal subgroups of $G$ and $M,M' \in \mathcal{M}$, we have $|M \cap \Pi| = |M' \cap \Pi|$.  

For a maximal subgroup $M \not\in \mathcal{C}$, we define $$d(M) = \sum\limits_{i \in I}\frac{|M \cap \Pi_i|}{|M_i \cap \Pi_i|},$$ where $M_i$ is a maximal subgroup in $\mathcal{M}_i$.  We will sometimes abuse notation slightly and write the isomorphism type of a group instead of a specific subgroup; for instance, $d(S_{13} \times S_{17})$ would denote $d(M)$, where $M$ is a maximal subgroup isomorphic to $S_{13} \times S_{17}$ in $S_{30}$.  The following lemma provides a sufficient condition for the cover $\mathcal{C}$ to be a minimal cover of the elements of $\Pi$.

\begin{lem}
\label{lem:keylemma}
Let $\Pi$ be a union of conjugacy classes of elements of $G$, and let $\mathcal{C} = \bigcup\limits_{i \in I} \mathcal{M}_i$ be a cover of $\Pi$ such that the elements of $\Pi$ are partitioned among the subgroups in $\mathcal{C}$ and that each subgroup in $\mathcal{C}$ contains elements of $\Pi$.  If $d(M) < 1$ for all maximal subgroups $M \not\in \mathcal{C}$, then $\mathcal{C}$ is a minimal cover of the elements of $\Pi$.  Moreover, $\mathcal{C}$ is the unique minimal cover of the elements of $\Pi$ that uses only maximal subgroups.
\end{lem}

\begin{proof}
Let $\mathcal{C}$ and $\Pi$ be as in the statement of the lemma, and assume that $d(M) < 1$ for all maximal subgroups not in $\mathcal{C}$.  Let $\mathcal{B}$ be another cover of the elements of $\Pi$.  Let $\mathcal{C}' = \mathcal{C} \backslash (\mathcal{C} \cap \mathcal{B})$ and $\mathcal{B}' = \mathcal{B} \backslash (\mathcal{C} \cap \mathcal{B})$.  The collection $\mathcal{C}'$ consists only of subgroups from classes $\mathcal{M}_i$, where $i \in I$, and we let $c_i$ be the number of subgroups from $\mathcal{M}_i$ in $\mathcal{C}'$.  Similarly, the collection $\mathcal{B}'$ consists only of subgroups from classes $\mathcal{M}_j$, where $j \not\in I$, and we let $b_j$ be the number of subgroups from $\mathcal{M}_j$ in $\mathcal{B}'$.  Note that, since $\mathcal{B}$ is a different cover, for some $j \not\in I$, we have $b_j > 0$.  

By removing $c_i$ subgroups from class $\mathcal{M}_i$ from $\mathcal{C}$, the new subgroups in $\mathcal{B}'$ must cover the elements of $\Pi$ that were in these subgroups.  Hence, for all $i \in I$,  if $M_k$ denotes a subgroup in class $\mathcal{M}_k$ for each $k$, $$c_i |M_i \cap \Pi_i| \le \sum\limits_{j \not\in I} b_j |M_j \cap \Pi_i|,$$ which in turn implies that, for all $i \in I$, $$c_i \le \sum\limits_{j \not\in I} b_j \frac{|M_j \cap \Pi_i|}{|M_i \cap \Pi_i|}.$$  This means that:
\begin{align*}
|\mathcal{C}'| &= \sum\limits_{i \in I} c_i\\
&\le \sum\limits_{i\in I}\sum\limits_{j \not\in I} b_j \frac{|M_j \cap \Pi_i|}{|M_i \cap \Pi_i|}\\
&= \sum\limits_{j \not\in I}\sum\limits_{i \in I} b_j \frac{|M_j \cap \Pi_i|}{|M_i \cap \Pi_i|}\\
&= \sum\limits_{j \not\in I}\left(\sum\limits_{i \in I} \frac{|M_j \cap \Pi_i|}{|M_i \cap \Pi_i|}  \right)b_j\\
&= \sum\limits_{j \not\in I} d(M_j)b_j\\
&< \sum\limits_{j \not\in I} b_j\\
&= |\mathcal{B}'|,
\end{align*}
which shows that $$|\mathcal{C}| = |\mathcal{C}'| + |\mathcal{C} \cap \mathcal{B}| < |\mathcal{C}'| + |\mathcal{C} \cap \mathcal{B}| = |\mathcal{B}|.$$ Hence, any other cover of the elements of $\Pi$ using only maximal subgroups has more subgroups than $\mathcal{C}$.  Therefore, $\mathcal{C}$ is a minimal cover of the elements of $\Pi$, and $\mathcal{C}$ is the unique minimal cover of the elements of $\Pi$ that uses only maximal subgroups. 
\end{proof}

Sometimes, an alternative formulation is easier to apply.  Let $\mathcal{M}_j$ be a class of maximal subgroups of $G$, $j \not\in I$.  For $i \in I$, we define $m_j(i)$ to be the number of subgroups in class $\mathcal{M}_j$ containing a particular element $g_i \in \Pi_i$.  For instance, if the elements of $\Pi_i$ of partitioned among the subgroups of $\mathcal{M}_j$, then $m_j(i) = 1$, whereas if each element of $\Pi_i$ is contained in exactly three subgroups of $\mathcal{M}_j$, then $m_j(i) = 3$.  We also define $$\epsilon_j(i) :=  \begin{cases}
1, & \text{ if } M_j \cap \Pi_i \neq \varnothing \text{ for all } M_j \in \mathcal{M}_j,\\
0, & \text{ if } M_j \cap \Pi_i = \varnothing \text{ for all } M_j \in \mathcal{M}_j.\\                                                                                                                                                                                                                                                                                                                                                                                                                                                                                                                                                                                                                                                                                                                                                             \end{cases}$$

\begin{lem}
\label{lem:altd}
Let $\Pi$ be a union of conjugacy classes of elements of $G$, and let $\mathcal{C} = \bigcup\limits_{i \in I} \mathcal{M}_i$ be a cover of $\Pi$ such that the elements of $\Pi$ are partitioned among the subgroups in $\mathcal{C}$ and that each subgroup in $\mathcal{C}$ contains elements of $\Pi$.  For any class $\mathcal{M}_j$ of maximal subgroups of $G$, $j \not\in I$, if $M_j \in \mathcal{M}_j$ and $M_i \in \mathcal{M}_i$, then $d(M_j) =  |M_j| \sum\limits_{i \in I} \frac{\epsilon_j(i)m_j(i)}{|M_i|}.$ 
\end{lem}

\begin{proof}
Since the elements of $\Pi_i$ are partitioned in $\mathcal{M}_i$, $$|M_i \cap \Pi_i| = \frac{|\Pi_i|}{|G:M_i|} = \frac{|\Pi_i| |M_i|}{|G|}.$$  
If $\epsilon_j(i) = 1$, then $$|M_j \cap \Pi_i| = \frac{|\Pi_i|}{|G:M_j|/m_j(i)} = \frac{|\Pi_i|m_j(i)|M_j|}{|G|} = \frac{|\Pi_i|\epsilon_j(i)m_j(i)|M_j|}{|G|}.$$  Similarly, if $\epsilon_j(i) = 0$, then $$|M_j \cap \Pi_i| = 0 = \epsilon_j(i) = \frac{|\Pi_i|\epsilon_j(i)m_j(i)|M_j|}{|G|}.$$  In any case, $$\frac{|M_j \cap \Pi_i|}{|M_i \cap \Pi_i|} = \frac{\left( \frac{|\Pi_i|\epsilon_j(i)m_j(i)|M_j|}{|G|}\right)}{\left( \frac{|\Pi_i| |M_i|}{|G|}\right)} = \frac{\epsilon_j(i)m_j(i)|M_j|}{|M_i|}.$$  The result follows.
\end{proof}

\section{Small values of n}
\label{sect:small}

In this section, we will exhibit the usefulness of Lemma \ref{lem:keylemma} by applying it to the cases $n = 18$ and $n = 24$.  This provides a concrete example of the lemma in action while simultaneously examining the two values of $n$ that are divisible by $6$ for which $\sigma(S_n)$ is unknown but are too small to be treated using the techniques in the next section. 

\subsection{The symmetric group \texorpdfstring{\boldmath{$S_{18}$}}{}}

We begin with the symmetric group $S_{18}$.  We first define the collection $\mathcal{C}_{18}$ to be the set of all maximal subgroups of $S_{18}$ isomorphic to one of $S_9 \Wr S_2$, $A_{18}$, $(S_1 \times) S_{17}$, $S_3 \times S_{15}$, $S_4 \times S_{14}$, or $S_5 \times S_{13}$.

\begin{lem}
\label{lem:18cover}
The collection $\mathcal{C}_{18}$ is a cover of the elements of $S_{18}$. 
\end{lem}

\begin{proof}
We identify $S_{18}$ with its natural action on the set $\{1,\dots, 18\}$.  We begin by noting that every $18$-cycle is contained in a subgroup isomorphic to $S_9 \Wr S_2$.  If $g \in S_{18}$ is a permutation that fixes an element of $\{1,\dots, 18\}$, then $g$ is contained in a subgroup isomorphic to $S_{17}$.  Moreover, any element with cycle structure $(i, 18- i)$, where $1 \le i \le 9$, is contained in $A_{18}$.  This means that any element $g$ not covered by a subgroup in $\mathcal{C}$ must have cycle structure consisting of at least three disjoint cycles, and $g$ cannot fix any element in the set $\{1,\dots, 18\}$.  If any one of the disjoint cycles of $g$ has length $3$, $4$, or $5$, then $g$ is contained in $S_3 \times S_{15}$, $S_4 \times S_{14}$, or $S_5 \times S_{13}$, respectively.  If all of the cycles of $g$ have length at least $6$, then the cycle structure of $g$ is $(6,6,6)$; any element with cycle structure $(6,6,6)$ stabilizes a decomposition of $\{1,\dots, 18\}$ into two subsets of size nine, and hence any element with cycle structure $(6,6,6)$ is contained in a subgroup isomorphic to $S_9 \Wr S_2$.  This implies that any element $g$ not covered by $\mathcal{C}$ has cycle structure $(2, i_1, i_2, \dots, i_k)$, where $\sum\limits_{j=1}^k i_j = 16$.  If some $i_j = 2$, then $g$ is contained in some subgroup isomorphic to $S_4 \times S_{14}$, so we may assume that each $i_j$ is at least $6$.  This implies that $k = 2$ and that $g$ has cycle structure one of $(2,6,10)$, $(2,7,9)$, or $(2,8,8)$.  However, all elements with one of these cycle structures stabilize a decomposition of $\{1,\dots, 18\}$ into two subsets of size $9$ and are contained in a subgroup isomorphic to $S_9 \Wr S_2$.  Therefore, $\mathcal{C}_{18}$ is a cover of $S_{18}$.
\end{proof}

We now need to show that $\mathcal{C}$ is in fact a minimal cover.  We define $\Pi$ to be the set of all elements of $S_{18}$ with cycle structure one of $(18)$, $(7,11)$, $(1, 7,10)$, $(3,7,8)$, $(4,7,7)$, or $(5,6,7)$.  Note that these elements are partitioned among the subgroups in $\mathcal{C}$.  We index the classes of maximal subgroups of $\mathcal{C}$ as follows: we let the subgroups isomorphic to $S_9 \Wr S_2$ be $\mathcal{M}_{-1}$, the subgroup isomorphic to $A_{18}$ be $\mathcal{M}_0$, and the subgroups isomorphic to $S_i \times S_{18-i}$ be $\mathcal{M}_i$ for $i = 1, 3,4,5$.  By our choice of indices, this means that $\Pi_{-1}$ is the set of $18$-cycles, $\Pi_0$ is the set of elements with cycle structure $(7,11)$, $\Pi_1$ is the set of elements with cycle structure $(1, 7,10)$, $\Pi_3$ is the set of elements with cycle structure $(3,7,8)$, $\Pi_4$ is the set of elements with cycle structure $(4,7,7)$, and $\Pi_5$ is the set of elements with cycle structure $(5,6,7)$.  The set $I$ is $\{-1,0,1,3,4,5\}$.  We will show that $\mathcal{C}_{18}$ is a minimal cover of the elements of $S_{18}$ by showing that $\mathcal{C}_{18}$ is a minimal cover of the elements of $\Pi$.

\begin{lem}
\label{lem:18min}
The collection $\mathcal{C}_{18}$ is a minimal cover of the elements of $\Pi$.  Moreover, $\mathcal{C}_{18}$ is the unique minimal cover of the elements of $\Pi$ using only maximal subgroups. 
\end{lem}

\begin{proof}
We begin by noting the number of elements of $\Pi_i$, $i \in I$, that are in each class of maximal subgroups.  We will start with the subgroups in $\mathcal{M}_i$ for some $i \in I$.  The subgroup isomorphic to $A_{18}$ contains $83147710464000$ elements of $\Pi_0$.  The subgroups isomorphic to $S_{17}$ each contain $5081248972800$ elements of $\Pi_1$.  The subgroups isomorphic to $S_3 \times S_{15}$ each contain $46702656000$ elements of $\Pi_3$.  The subgroups isomorphic to $S_4 \times S_{14}$ each contain $5337446400$ elements of $\Pi_4$.  The subgroups isomorphic to $S_5 \times S_{13}$ each contain $3558297600$ elements of $\Pi_5$.  Finally, the subgroups isomorphic to $S_9 \Wr S_2$ each contain $14631321600$ elements of $\Pi_{-1}$.  

We will now calculate $d(M)$ for the maximal subgroups not contained in $\mathcal{C}$.  Beyond the classes contained in $\mathcal{C}_{18}$, the maximal subgroups of $S_{18}$ are isomorphic to one of the following: $S_2 \times S_{16}$, $S_6 \times S_{12}$, $S_7 \times S_{11}$, $S_8 \times S_{10}$, $S_6 \Wr S_3$, $S_3 \Wr S_6$, $S_2 \Wr S_9$, or $\PGL_2(17)$.

First, since the subgroups isomorphic to $S_2 \times S_{16}$ contain no elements of $\Pi$, $$d(S_2 \times S_{16}) = 0 < 1.$$

The subgroups isomorphic to $S_6 \times S_{12}$ contain only $1642291200$ elements of $\Pi_5$.  This means that, if $M_6$ is a subgroup isomorphic to $S_6 \times S_{12}$ and $M_5$ is a subgroup isomorphic to $S_5 \times S_{13}$, $$d(S_6 \times S_{12}) = \frac{|M_6 \cap \Pi_5|}{|M_5 \cap \Pi_5|} = \frac{1642291200}{3558297600} < 1.$$

The subgroups isomorphic to $S_7 \times S_{11}$ each contain $2612736000$ elements from $\Pi_0$,\\ $2874009600$ elements from $\Pi_1$, $1197504000$ elements from $\Pi_3$, $102643200$ elements from $\Pi_4$, and $958003200$ elements from $\Pi_5$.  This means that \begin{align*}d(S_7 \times S_{11}) = \text{ }&\frac{2612736000}{83147710464000} + \frac{2874009600}{5081248972800} + \frac{1197504000}{46702656000}\\ \text{   }&+ \frac{102643200}{5337446400} + \frac{958003200}{3558297600}\\ < \text{ }&1.\end{align*}

The subgroups isomorphic to $S_8 \times S_{10}$ each contain $2090188800$ elements from $\Pi_1$ and $870912000$ elements from $\Pi_3$.  This means that $$d(S_8 \times S_{10}) = \frac{2090188800}{5081248972800} + \frac{870912000}{46702656000} < 1.$$

The subgroups isomorphic to $S_6 \Wr S_3$, $S_3 \Wr S_6$, $S_2 \Wr S_9$, and $\PGL_2(17)$ only contain elements from $\Pi_{-1}$.  In each case, we have:

\begin{align*}
d(S_6 \Wr S_3) &= \frac{12441600}{14631321600} < 1,\\ 
d(S_3 \Wr S_6) &= \frac{1866240}{14631321600} < 1,\\
d(S_2 \Wr S_9) &= \frac{10321920}{14631321600} < 1,\\
d(\PGL_2(17)) &= \frac{816}{14631321600} < 1.\\
\end{align*}
Therefore, for all maximal subgroups $M$ of $S_{18}$ not contained in $\mathcal{C}$, we have $d(M) < 1$.  By Lemma \ref{lem:keylemma}, $\mathcal{C}$ is the unique minimal cover of $\Pi$ using only maximal subgroups.
\end{proof}

\begin{thm}
\label{thm:18}
The covering number of $S_{18}$ is $36772$, and the unique minimal cover containing only maximal subgroups consists of all subgroups isomorphic to one of $S_9 \Wr S_2$, $A_{18}$, $(S_1 \times) S_{17}$, $S_3 \times S_{15}$, $S_4 \times S_{14}$, or $S_5 \times S_{13}$. 
\end{thm}

\begin{proof}
This follows immediately from Lemmas \ref{lem:18cover} and \ref{lem:18min}. 
\end{proof}

\subsection{The symmetric group \texorpdfstring{\boldmath{$S_{24}$}}{}}

We proceed with $S_{24}$ as we did with $S_{18}$ above.  We define the collection $\mathcal{C}_{24}$ to be the set of all maximal subgroups of $S_{24}$ isomorphic to one of $S_{12} \Wr S_2$, $A_{24}$, or $S_i \times S_{24-i}$, where $1 \le i \le 7$.  

\begin{lem}
\label{lem:24cover}
The collection $\mathcal{C}_{24}$ is a cover of the elements of $S_{24}$.  
\end{lem}

\begin{proof}
We identify $S_{24}$ with its natural action on $\{1,\dots,24\}$ and consider an element $g \in S_{24}$.  If $g$ fixes any elements in $\{1,\dots,24\}$, then $g$ is contained in a subgroup isomorphic to $S_{23}$.  If $g$ is an $24$-cycle, then $g$ is contained in a subgroup isomorphic to $S_{12} \Wr S_2$.  If $g$ has cycle structure $(j, 24 - j)$ for some $1 \le j \le 12$, then $g$ is contained in $A_{24}$.  If the cycle structure of $g$ contains an $i$-cycle, where $2 \le i \le 7$, then $g$ is contained in a subgroup isomorphic to $S_i \times S_{24 - i}$.  Thus any element $g$ not covered by $\mathcal{C}_{24}$ must fix no points of $\{1,\dots,24\}$ and have cycle structure consisting of at least three disjoint cycles whose lengths are all at least $8$.  The only such elements $g$ have cycle structure $(8,8,8)$.  However, these elements are contained in the subgroups isomorphic $S_{12} \Wr S_2$, and, therefore, $\mathcal{C}_{24}$ is a cover of the elements of $S_{24}$.   
\end{proof}

To show that $\mathcal{C}_{24}$ is a minimal cover, we consider the set $\Pi$, which consists of all elements of $S_{24}$ with cycle structure one of $(24)$, $(11,13)$, $(1,10,13)$, $(2,11,11)$, $(3,10,11)$, $(4,9,11)$, $(5,9,10)$, $(6,9,9)$, or $(7,8,9)$.  We note that the elements of $\Pi$ are partitioned among the subgroups in $\mathcal{C}_{24}$.  We index the subgroups of $\mathcal{C}_{24}$ as follows: we let $\mathcal{M}_{-1}$ be the class of subgroups isomorphic to $S_{12} \Wr S_2$; we let $\mathcal{M}_0$ be the class containing the subgroup $A_{24}$; and, for $1 \le i \le 7$, we let $\mathcal{M}_i$ be the class of subgroups isomorphic to $S_i \times S_{24-i}$.  We let $I = \{-1, 0, \dots,7\}$.  By our choice of indices, this means that $\Pi_{-1}$ is the set of elements with cycle structure $(24)$, $\Pi_0$ is the set of elements with cycle structure $(11,13)$, and $\Pi_i$ is the set of elements of $\Pi$ containing an $i$-cycle, where $1 \le i \le 7$.

\begin{lem}
\label{lem:24min}
The collection $\mathcal{C}_{24}$ is a minimal cover of the elements of $\Pi$.  Moreover, $\mathcal{C}_{24}$ is the unique minimal cover of the elements of $\Pi$ using only maximal subgroups.
\end{lem}

\begin{proof}
We will calculate $d(M)$ for each maximal subgroup $M$ not in $\mathcal{C}$.  These maximal subgroups $M$ must be isomorphic to one of the following: $S_j \times S_{24-j}$, where $8 \le j \le 11$; $S_k \Wr S_{n/k}$, where $k = 2,3,4,6,8$; or $\PGL_2(23)$.  We leave out the details of the calculations but present the values of each $d(M)$:

\begin{align*}
d(S_8 \times S_{16}) &= \frac{8}{17} < 1,\\
d(S_9 \times S_{15}) &= \frac{1321}{2584} < 1,\\
d(S_{10} \times S_{14}) &= \frac{256}{245157} < 1,\\
d(S_{11} \times S_{13}) &= \frac{22441}{832048} < 1,\\
d(S_8 \Wr S_3) &= \frac{14}{16335} < 1,\\
d(S_6 \Wr S_4) &= \frac{1}{71148} < 1,\\
d(S_4 \Wr S_6) &= \frac{2}{6670125} < 1,\\
d(S_3 \Wr S_8) &= \frac{243}{1497496060} < 1,\\
d(S_2 \Wr S_{12}) &= \frac{2}{467775} < 1,\\
d(\PGL_2(23)) &= \frac{23}{217275125760000} < 1.\\
\end{align*}
Therefore, by Lemma \ref{lem:keylemma}, $\mathcal{C}_{24}$ is the unique minimal cover of the elements $\Pi$ using only maximal subgroups.
\end{proof}

\begin{thm}
\label{thm:24}
The covering number of $S_{24}$ is $1888233$, and the unique minimal cover containing only maximal subgroups consists of all subgroups isomorphic to one of $S_{12} \Wr S_2$, $A_{24}$, or $S_i \times S_{n-i}$, where $1 \le i \le 7$.. 
\end{thm}

\begin{proof}
This follows immediately from Lemmas \ref{lem:24cover} and \ref{lem:24min}. 
\end{proof}

\section{Large values of n}
\label{sect:large}

In this section, we determine the covering number of $S_n$, where $n \ge 30$ and $n \equiv 0 \pmod 6$.

We define the collection $\mathcal{C}_n$ to consist of all maximal subgroups of $S_n$ isomorphic to one of the following: $S_{n/2} \Wr S_2$, $A_n$, or $S_i \times S_{n-i}$, where $1 \le i \le n/3 - 1$.  We label the classes of maximal subgroups as follows: $\mathcal{M}_{-1}$ is the class of subgroups isomorphic to $S_{n/2} \Wr S_2$; $\mathcal{M}_0$ is the class that contains $A_{n}$; and, for $1 \le i \le n/3 - 1$, $\mathcal{M}_i$ is the class that contains the subgroups isomorphic to $S_i \times S_{n-i}$.  We let $I = \{-1,0,\dots, n/3 - 1\}$. 

\begin{lem}
\label{lem:Cncover}
Let $n \equiv 0 \pmod 6$ and $n \ge 30$.  The collection $\mathcal{C}_n$ is a cover of the elements of $S_n$. 
\end{lem}

\begin{proof}
We identify $S_{n}$ with its natural action on $\{1,\dots,n\}$ and consider an element $g \in S_{n}$.  If $g$ fixes any element in $\{1,\dots,n\}$, then $g$ is contained in a subgroup isomorphic to $S_{n-1}$.  If $g$ is an $n$-cycle, then $g$ preserves a decomposition of $\{1,\dots,n\}$ into two sets of size $n/2$, and hence $g$ is contained in a subgroup isomorphic to $S_{n/2} \Wr S_2$.  If $g$ has cycle structure $(j, n - j)$ for some $1 \le j \le n/2$, then $g$ is contained in $A_{n}$.  If the cycle structure of $g$ contains an $i$-cycle, where $2 \le i \le n/3 - 1$, then $g$ is contained in a subgroup isomorphic to $S_i \times S_{n - i}$.  Thus any element $g$ not covered by $\mathcal{C}_{n}$ must fix no points of $\{1,\dots,n\}$ and have cycle structure consisting of at least three disjoint cycles whose lengths are all at least $n/3$.  The only such elements $g$ have cycle structure $(n/3,n/3,n/3)$.  However, since $n/3$ is even, these elements all stabilize a decomposition of $\{1,\dots,n\}$ into two sets of size $n/2$ and are contained in the subgroups isomorphic $S_{n/2} \Wr S_2$.  Therefore, $\mathcal{C}_{n}$ is a cover of the elements of $S_{n}$. 
\end{proof}

We now define collections $\Pi_i$, $-1 \le i \le n/3 - 1$, as follows:

\begin{align*}
 \Pi_{-1} := \{&g : g \in S_n, g \text{ has cycle structure } (n)\};\\
 \Pi_{0} := \{&g : g \in S_n, g \text{ has cycle structure } (n/2 - 1, n/2 + 1) \text{ if } n/2 \text{ is even }\\
  &\text{or cycle structure } (n/2 -2, n/2+2) \text{ if } n/2 \text{ is odd}\};\\
 \Pi_{1} := \{&g : g \in S_n, g \text{ has cycle structure } (1, n/2 - 2, n/2 + 1)\};\\
 \Pi_{2} := \{&g : g \in S_n, g \text{ has cycle structure } (2, n/2 - 1, n/2 - 1) \text{ if } n/2 \text{ is even}\\
 &\text{or cycle structure } (2, n/2 -4, n/2+2) \text{ if } n/2 \text{ is odd}\};\\
 \Pi_{i} := \{&g : g \in S_n, g \text{ has cycle structure } (i, \lfloor(n-i)/2\rfloor, \lceil (n-i)/2 \rceil)\},\\ &\text { for } 3 \le i \le \frac{n}{3}-1 \text{ odd};\\
 \Pi_{i} := \{&g : g \in S_n, g \text{ has cycle structure } (i, (n-i)/2, (n-i)/2) \text{ if } (n-i)/2 \text{ is odd}\\
 &\text{or cycle structure } (i, (n-i)/2 - 1, (n-i)/2 + 1) \text{ if } (n-i)/2 \text{ is even}\},\\
 &\text{ for } 4 \le i \le \frac{n}{3}-2 \text{ even}.\\
\end{align*}

We let $\Pi = \bigcup\limits_{i \in I} \Pi_i.$  For $1 \le i \le n/3 - 1$, we will denote the cycle structure of elements in $\Pi_i$ by $(i, r_i, s_i)$, where $r_i \le s_i$.

\begin{lem}
\label{lem:partitioned}
Let $n \equiv 0 \pmod 6$ and $n \ge 30$.  For each $i$, $-1 \le i \le n/3 - 1$, the only subgroups in $\mathcal{C}_n$ that contain elements of $\Pi_i$ are in class $\mathcal{M}_i$.  Moreover, the elements of $\Pi$ are partitioned among the subgroups in $\mathcal{C}_n$. 
\end{lem}

\begin{proof}
We identify $S_n$ with its natural action on $\{1, \dots, n\}$.  The elements in $\Pi_{-1}$ are $n$-cycles. The $n$-cycles are odd permutations, since $n$ is even, and so they are not contained in $A_{n}$, the unique subgroup in $\mathcal{M}_0$.  Moreover, the $n$-cycles are transitive on $\{1, \dots, n\}$ and cannot be contained in any subgroup isomorphic to $S_i \times S_{n-i}$, where $1 \le i \le n/3 - 1$.  Each $n$-cycle stabilizes a unique partition of $\{1, \dots, n\}$ into two sets of size $n/2$, and so the $n$-cycles are partitioned among the subgroups in $\mathcal{M}_{-1}$.

The elements in $\Pi_0$ have cycle structure $(n/2 - 1, n/2 + 1)$ if $n/2$ is even or $(n/2 -2, n/2 + 2)$ if $n/2$ is odd.  These elements are even permutations and are contained in $\mathcal{M}_0$.  In either case, these elements have a cycle structure that contains a cycle of odd length and a cycle that is longer than $n/2$.  This means they cannot stabilize a decomposition of $\{1, \dots, n\}$ into two sets of size $n/2$ and cannot be contained in the subgroups isomorphic to $S_{n/2} \Wr S_2$.  Moreover, these elements do not stabilize a decomposition of $\{1, \dots, n\}$ into a set of size $i$ and a set of size $n-i$, where $1 \le i \le n/3 - 1$.  Hence $A_n$, the unique subgroup of $\mathcal{M}_0$, contains all elements in $\Pi_0$ and is the only subgroup of $\mathcal{C}_n$ to contain elements of $\Pi_0$.

Finally, we consider the elements of $\Pi_i$ for $1 \le i \le n/3 - 1$, which have cycle structure $(i, r_i, s_i)$.  These elements stabilize a unique decomposition of $\{1, \dots, n\}$ into a set of size $i$ and a set of size $r_i + s_i = n-i$, and so they are partitioned among the subgroups in $\mathcal{M}_i$.  By construction, at least one of $i$, $r_i$, or $s_i$ is odd, and so the only way that an element with cycle structure $(i,r_i, s_i)$ could stabilize a partition of $\{1, \dots, n\}$ into two sets of size $n/2$ is if either $i + r_i$ or $i + s_i$ is $n/2$, which implies, respectively, that $s_i$ or $r_i$ is $n/2$.  Since $r_i \le s_i$, we can rule out $n/2 = i + s_i = r_i$.  This means that $s_i = n/2$, but $s_i \neq n/2$ by the definition of the $\Pi_i$.  Hence the subgroups isomorphic to $S_{n/2} \Wr S_2$ in $\mathcal{M}_{-1}$ do not contain any elements from $\Pi_i$, $1 \le i \le n/3 - 1$.  Elements with cycle structure $(i,r_i, s_i)$, where $i + r_i + s_i = n$, are odd permutations, and so they are not contained in $A_n$, the unique subgroup in $\mathcal{M}_0$.  Finally, the only subgroups isomorphic to $S_j \times S_{n-j}$ that contain elements with cycle structure $(i,r_i, s_i)$ have $j$ equal to one of $i$, $r_i$, $s_i$, or $i+r_i$.  However, each of $r_i$, $s_i$ are at least $n/3$, so the only subgroups $S_j \times S_{n-j}$ with $j \le n/3 - 1$ containing elements with cycle structure $(i,r_i, s_i)$ are isomorphic to $S_i \times S_{n-i}$.  Therefore, for each $i$, $-1 \le i \le n/3 - 1$, the only subgroups in $\mathcal{C}_n$ that contain elements of $\Pi_i$ are in class $\mathcal{M}_i$, and the elements of $\Pi$ are partitioned among the subgroups in $\mathcal{C}_n$.
\end{proof}

In order to apply Lemma \ref{lem:keylemma}, we now must show that $d(M) < 1$ for all maximal subgroups $M \not\in \mathcal{C}_{n}$. 

\begin{lem}
\label{lem:eltbound}
Let $M_i \in \mathcal{M}_i$, where $-1 \le i \le n/3 - 1$.  If $30 \le n \le 102$, then $|M_i \cap \Pi_i| \ge (n/3 - 2)!\frac{(2n/3 +1)!}{(n/3 + 1)n/3}$.  If $n \ge 108$, then $|M_i \cap \Pi_i| \ge (n/2 - 1)!(n/2)!$.
\end{lem}

\begin{proof}
The alternating group $A_n$ contains $n!/((n/2-1)(n/2+1))$ different elements of $\Pi$ when $n/2$ is even and $n!/((n/2-2)(n/2+2))$ different elements of $\Pi$ when $n/2$ is odd.  The subgroups isomorphic to $S_{n/2} \Wr S_2$ each contain $(n/2 - 1)!(n/2)!$ different $n$-cycles.  The subgroups isomorphic to $S_i \times S_{n-i}$, where $1 \le i \le n/3 - 1$, each contain $(i-1)!{\binom{n-i}{r_i}}(r_i - 1)!(s_i - 1)!$ elements with cycle structure $(i,r_i,s_i)$ if $r_i \neq s_i$ and $\frac{1}{2}(i-1)!{\binom{n-i}{r_i}}(r_i - 1)!^2 $ elements of $\Pi$ if $r_i = s_i$.  For $1 \le i \le n/3 - 1$, $|M_i \cap \Pi_i|$ is at least $(n/3 - 2)!{\binom{2n/3 + 1}{n/3}}(n/3 - 1)!(n/3)!$.  The result follows by comparing the values of $(n/3 - 2)!{\binom{2n/3 + 1}{n/3}}(n/3 - 1)!(n/3)! = (n/3 - 2)!\frac{(2n/3 +1)!}{(n/3 + 1)n/3}$ and $(n/2 - 1)!(n/2)!$ for all $n \equiv 0 \pmod 6$ and $n \ge 30$.
\end{proof}

\begin{lem}
\label{lem:dprim} 
Let $n \equiv 0 \pmod 6$ and $n \ge 30$.  If $M$ is a primitive maximal subgroup of $S_n$ that is not contained in $\mathcal{C}_n$, then $d(M) < 1$.  
\end{lem}

\begin{proof}
Let $M$ be a primitive maximal subgroup of $S_n$ that is not contained in $\mathcal{C}_n$, i.e., a primitive maximal subgroup that is not isomorphic to $A_n$.  By Lemma \ref{lem:primbound}, for all $i \in I$ we know that $$|M \cap \Pi_i| < |M| < 2^n.$$  By Lemma \ref{lem:eltbound}, when $30 \le n \le 102$, for all $i \in I$ and maximal subgroups $M_i \in \mathcal{M}_i$, we have $|M_i \cap \Pi_i| \ge  (n/3 - 2)!\frac{(2n/3 +1)!}{(n/3 + 1)n/3}$.  Hence, when $30 \le n \le 102$, 
\begin{align*}
d(M) &= \sum\limits_{i \in I} \frac{|M \cap \Pi_i|}{|M_i \cap \Pi_i|}\\
&< 2^n \cdot \frac{|I|}{\min\limits_{i \in I} |M_i \cap \Pi_i|}\\
&\le \frac{2^n(n/3 + 1)}{(n/3 - 2)!\frac{(2n/3 +1)!}{(n/3 + 1)n/3}}\\
&< 1.\\
\end{align*}
Similarly, when $n \ge 108$, by Lemma \ref{lem:eltbound}, for all $i \in I$ and maximal subgroups $M_i \in \mathcal{M}_i$, we have $|M_i \cap \Pi_| \ge (n/2 - 1)!(n/2)!$, which implies that
\begin{align*}
d(M) &= \sum\limits_{i \in I} \frac{|M \cap \Pi_i|}{|M_i \cap \Pi_i|}\\
&< 2^n \cdot \frac{|I|}{\min\limits_{i \in I} |M_i \cap \Pi_i|}\\
&\le \frac{2^n(n/3 + 1)}{(n/2 - 1)!(n/2)!}\\
&< 1.\\
\end{align*}
In any case, $d(M) < 1$ for all such maximal subgroups $M$.
\end{proof}

\begin{lem}
\label{lem:dimprim}
Let $n \equiv 0 \pmod 6$ and $n \ge 30$.  If $M$ is a transitive, imprimitive maximal subgroup of $S_n$ that is not contained in $\mathcal{C}_n$, then $d(M) < 1$.
\end{lem}

\begin{proof}
Let $M$ be a transitive, imprimitive maximal subgroup that is not contained in $S_n$, i.e., $M$ is isomorphic to $S_{n/k} \Wr S_{k}$ for some divisor $k$ of $n$, where $k > 2$.  We note first that, for any $i \in I$ and $n \ge 30$, $$|M \cap \Pi_i| < |M| \le 6(n/3)!^3.$$  By Lemma \ref{lem:eltbound}, when $30 \le n \le 102$, for all $i \in I$ and maximal subgroups $M_i \in \mathcal{M}_i$, we have $|M_i \cap \Pi_i| \ge  (n/3 - 2)!\frac{(2n/3 +1)!}{(n/3 + 1)n/3}$.  Hence, when $30 \le n \le 102$, 
\begin{align*}
d(M) &= \sum\limits_{i \in I} \frac{|M \cap \Pi_i|}{|M_i \cap \Pi_i|}\\
&< 6(n/3)!^3 \cdot \frac{|I|}{\min\limits_{i \in I} |M_i \cap \Pi_i|}\\
&\le \frac{6(n/3)!^3(n/3 + 1)}{(n/3 - 2)!\frac{(2n/3 +1)!}{(n/3 + 1)n/3}}\\
&< 1.\\
\end{align*}
Similarly, when $n \ge 108$, by Lemma \ref{lem:eltbound}, for all $i \in I$ and maximal subgroups $M_i \in \mathcal{M}_i$, we have $|M_i \cap \Pi_| \ge (n/2 - 1)!(n/2)!$, which implies that
\begin{align*}
d(M) &= \sum\limits_{i \in I} \frac{|M \cap \Pi_i|}{|M_i \cap \Pi_i|}\\
&< 6(n/3)!^3 \cdot \frac{|I|}{\min\limits_{i \in I} |M_i \cap \Pi_i|}\\
&\le \frac{6(n/3)!^3(n/3 + 1)}{(n/2 - 1)!(n/2)!}\\
&< 1.\\
\end{align*}
In any case, $d(M) < 1$ for all such maximal subgroups $M$. 
\end{proof}

We must now consider the intransitive maximal subgroups of $S_n$ that are not in $\mathcal{C}_n$, i.e., those maximal subgroups isomorphic to $S_j \times S_{n-j}$ for some $n/3 \le j < n/2$.  In order to do this,  we will use the equivalent formula for $d(M)$ from Lemma \ref{lem:altd}.  

\begin{lem}
\label{lem:ijineq}
Let $n \equiv 0 \pmod 6$ and $n \ge 30$.  If a maximal subgroup of $S_n$ isomorphic to $S_j \times S_{n-j}$ contains elements of $\Pi_i$, where $n/3 \le j < n/2$ and $3 \le i \le n/3 - 1$, then $n - 2j - 2 \le i \le n - 2j + 2$.   
\end{lem}

\begin{proof}
We use the notation above and note that the elements in $\Pi_i$ have cycle structure $(i,r_i,s_i)$, where $r_i \le s_i$ and $i + r_i + s_i = n$.  Since $i \ge 3$, by definition, $(n-i)/2 - 1 \le r_i \le s_i \le (n-i)/2 + 1$.  Note that this implies that $s_i - r_i \le 2$.  Identifying $S_n$ with its natural action on $\{1, \dots, n\}$, any element in $S_j \times S_{n-j}$ must stabilize a decomposition of $\{1, \dots, n\}$ into a set of size $j$ and a set of size $n-j$.  Since $s_i - r_i \le 2$ and $n/3 \le j < n - j$, either $j = r_i$ and $n-j = s_i + i$ or $j = s_i$ and $n-j = r_i + i$.  In either case, this implies that $(n-i)/2 - 1 \le j \le (n-i)/2 + 1$.  The result follows. 
\end{proof}

\begin{lem}
\label{lem:epsilonji}
Let $n \equiv 0 \pmod 6$ and $n \ge 30$, and let $\mathcal{M}_j$ be the class of maximal subgroups isomorphic to $S_j \times S_{n-j}$, where $n/3 \le j < n/2$.  If $n/2$ is odd and $j = n/2 - 2$, then there are exactly seven $i \in I$ such that $\epsilon_j(i) = 1$.  If $n/2$ is odd, $n > 30$, and $j = n/2 - 4$, then there are exactly six $i \in I$ such that $\epsilon_j(i) = 1$.  Otherwise, there are at most five $i \in I$ such that $\epsilon_j(i) = 1$.
\end{lem}

\begin{proof}
Let $\mathcal{M}_j$ be the class of maximal subgroups of $S_n$ isomorphic to $S_j \times S_{n-j}$, where $n/3 \le j < n/2$.  Identifying $S_n$ with its natural action on $\{1, \dots, n \}$, we note that the subgroups of $\mathcal{M}_j$ are intransitive on $\{1, \dots, n\}$. Hence no $n$-cycles are contained in these subgroups and $\epsilon_j(-1) = 0$. By Lemma \ref{lem:ijineq}, there are at most five values of $i$, $3 \le i \le n/3 - 1$, for which $\epsilon_j(i) = 1$.  This means that, if there are more than five values of $i$ for which $\epsilon_j(i) = 1$, then $\epsilon_j(k) = 1$ for at least one value of $k$ in $\{0,1,2\}$.  Hence, unless $j = n/2 - 4,$ $n/2 - 2$, or $n/2 - 1$, we have $\epsilon_j(i) = 1$ for at most five values of $i$.

We now consider the remaining cases individually.  For $j = n/2 - 1$, by Lemma \ref{lem:ijineq}, $\epsilon_j(i) = 1$ only if $i \le 4$.  If $j = n/2 - 2$, by Lemma \ref{lem:ijineq}, then $\epsilon_j(i) = 1$ only if $i \le 6$.  In this case, if $n/2$ is even, then $\epsilon_j(i) = 1$ when $i = 1,3,5$; if $n/2$ is odd, then $\epsilon_j(i) = 1$ when $0 \le i \le 6$.  If $j = n/2 - 4$, by Lemma \ref{lem:ijineq}, then $\epsilon_j(i) = 1$ only if $0 \le i \le 2$ or $6 \le i \le 10$.  In this case, if $n/2$ is even, then $\epsilon_j(i) = 1$ when $i = 7,9$; if $n/2$ is odd, then $\epsilon_j(i) = 1$ when $i = 2$ and when $6 \le i \le 10$.  The result follows.
\end{proof}

\begin{lem}
\label{lem:mji}
Let $n \equiv 0 \pmod 6$ and $n \ge 30$, and let $\mathcal{M}_j$ be the class of maximal subgroups isomorphic to $S_j \times S_{n-j}$, where $n/3 \le j < n/2$.  There is at most one value of $i \in I$ for which $m_j(i) > 1$, namely $i = n - 2j$, and $m_j(n-2j) = 2$.
\end{lem}

\begin{proof}
We will identify $S_n$ with its natural action on $\{1, \dots, n\}$. Assume that the subgroups isomorphic to $S_j \times S_{n-j}$ contain elements with cycle structure $(i, r_i, s_i)$, where $r_i \le s_i$ and $i + r_i + s_i = n$.  When $i \le 2$, the only time $m_j(i) > 1$ is when $n/2$ is even, $i = 2$, and $j = n/2 - 1$, in which case $m_{n/2 - 1}(2) = 2$ and $m_{n/2 - 1}(i) \le 1$ for all other values of $i$.  When $i \ge 3$, proceeding as in the proof of Lemma \ref{lem:ijineq}, we see that either $j = r_i$ or $j = s_i$ when $i \ge 3$.  If $r_i \neq s_i$, then an element with cycle structure $(i,r_i,s_i)$ will stabilize a unique decomposition of $\{1, \dots, n \}$ into a set of size $j$ and a set of size $n-j$, which implies that such an element is contained in a unique subgroup in $\mathcal{M}_j$ and $m_j(i) = 1$.  Hence $m_j(i) > 1$ only if $r_i = s_i = j$, in which case $i = n - r_i - s_i = n - 2j$ is unique, and $m_j(n - 2j) = 2$. 
\end{proof}

\begin{lem}
\label{lem:intransspecial}
Let $n \equiv 0 \pmod 6$ and $n \ge 30$.  If $n/2$ is odd, then $d(S_j \times S_{n-j}) < 1$ when $j = n/2 - 2$ and $n/2 - 4$.
\end{lem}

\begin{proof}
For the case $j = n/2 - 2$, by Lemmas \ref{lem:altd}, \ref{lem:epsilonji}, and \ref{lem:mji}, we have:
\begin{align*}
d(S_{n/2 - 2} \times S_{n/2 + 2}) = & \text{ }|S_{n/2 -2} \times S_{n/2 + 2}| \left( \frac{1}{|A_n|} + \frac{1}{|S_{n-1}|} + \frac{1}{|S_2 \times S_{n-2}|}\right. \\ 
 & \left. + \frac{1}{|S_{3} \times S_{n-3}|} + \frac{2}{|S_4 \times S_{n-4}|} + \frac{1}{|S_5 \times S_{n-5}|} + \frac{1}{|S_6 \times S_{n-6}|} \right)\\
= &\text{ } (n/2 - 2)! (n/2+2)! \left( \frac{2}{n!} + \frac{1}{(n-1)!} + \frac{1}{2(n-2)!} \right. \\
&\left. + \frac{1}{6(n-3)!} + \frac{2}{24(n-4)!} + \frac{1}{120(n-5)!} + \frac{1}{720(n-6)!} \right)\\
<& 1,
\end{align*}
when $n \ge 30$.  

Now, consider $j = n/2 - 4$.  Assuming that $10 \le n/3 - 1$, $n \equiv 0 \pmod 6$, and $n/2$ is odd, we have $n \ge 42$.  By Lemmas \ref{lem:altd}, \ref{lem:epsilonji}, and \ref{lem:mji}, we have:
\begin{align*}
d(S_{n/2 - 4} \times S_{n/2 + 4}) = &\text{ } |S_{n/2 -4} \times S_{n/2 + 4}| \left( \frac{1}{|S_2 \times S_{n-2}|} + \frac{1}{|S_{6} \times S_{n-6}|} \right. \\
&\left. + \frac{1}{|S_7 \times S_{n-7}|} + \frac{2}{|S_8 \times S_{n-8}|} + \frac{1}{|S_9 \times S_{n-9}|} + \frac{1}{|S_{10} \times S_{n-10}|} \right)\\
= &\text{ } (n/2 - 4)! (n/2+4)! \left( \frac{1}{2(n-2)!} + \frac{1}{6!(n-6)!} \right. \\
&\left. + \frac{1}{7!(n-7)!} + \frac{2}{8!(n-8)!} + \frac{1}{9!(n-9)!} +  \frac{1}{10!(n-10)!}\right)\\
<& 1.
\end{align*}
\end{proof}

\begin{lem}
\label{lem:dintrans}
 Let $n \equiv 0 \pmod 6$ and $n \ge 30$.  If $n/3 \le j < n/2$, then $d(S_j \times S_{n-j}) < 1$.
\end{lem}

\begin{proof}
We denote by $\mathcal{M}_j$ the class of maximal subgroups of $S_n$ isomorphic to $S_j \times S_{n-j}$.  By Lemmas \ref{lem:epsilonji} and \ref{lem:intransspecial}, $d(S_j \times S_{n-j}) < 1$ whenever there are more than five values of $i \in I$ such that $\epsilon_j(i) = 1$.  Hence we may assume that $\epsilon_j(i) = 1$ for at most five values of $i$.  Moreover, by Lemma \ref{lem:mji}, there is at most one value of $i$ for which $m_j(i) > 1$, and, if $m_j(i) > 1$, then $m_j(i) = 2$.  Thus we may assume that $$\sum\limits_{i \in I} \epsilon_j(i)m_j(i) \le 6.$$  Letting $M_j$ denote a subgroup isomorphic to $S_j \times S_{n-j}$, using Lemma \ref{lem:altd} and noting that $\epsilon_j(-1) = 0$, we have: 
\begin{align*}
d(S_j \times S_{n-j}) &= |M_j| \sum\limits_{i \in I} \frac{\epsilon_j(i)m_j(i)}{|M_i|}\\
&= |M_j| \sum\limits_{i \in I, i \neq -1} \frac{\epsilon_j(i)m_j(i)}{|M_i|}\\
&\le \frac{|M_j|}{\min\limits_{i \in I, i \neq -1}|M_i|} \sum\limits_{i \in I, i \neq -1} \epsilon_j(i)m_j(i)\\
&\le \frac{6|M_j|}{\min\limits_{i \in I, i \neq -1}|M_i|}.
\end{align*}
Note that $\min\limits_{i \in I, i \neq -1}|M_i| = |S_{n/3} \times S_{2n/3 + 1}| = (n/3 -1)! (2n/3 + 1)!$.  Furthermore, when $n \ge 30$,
\begin{align*}
d(S_{n/3 + 2} \times S_{2n/3 - 2}) &\le \frac{6(n/3+2)!(2n/3 - 2)!}{(n/3 - 1)!(2n/3 + 1)!}\\
&= \frac{6(n/3 + 2)(n/3 + 1)(n/3)}{(2n/3  +1)(2n/3)(2n/3 - 1)}\\
&= \frac{3n^2 + 27n + 54}{4n^2 - 9}\\
&< 1.\\
\end{align*}
If $n/3 < j_1 < j_2 < n/2$, then $|S_{j_1} \times S_{n-j_1}| > |S_{j_2} \times S_{n -j_2}|$ and $d(S_{j_2} \times S_{n- j_2}) < d(S_{j_1} \times S_{n - j_1})$.  Thus $d(S_j \times S_{n-j}) < 1$ when $n \ge 30$ and $j \ge n/3 + 2$, and we need only check $j = n/3$ and $j = n/3 + 1$.

We consider first the case $j = n/3$.  By the definition of the permutations in $\Pi_i$ and Lemma \ref{lem:ijineq}, $\epsilon_{n/3}(i) = 1$ only when $i = n/3 - 1$, and so we have:
\begin{align*}
d(S_{n/3} \times S_{2n/3}) &= \frac{m_{n/3}(n/3 - 1)|S_{n/3} \times S_{2n/3}|}{|S_{n/3 -1} \times S_{2n/3 + 1}|}\\ 
&= \frac{|S_{n/3} \times S_{2n/3}|}{|S_{n/3 -1} \times S_{2n/3 + 1}|}\\ 
&= \frac{(n/3)!(2n/3)!}{(n/3 - 1)!(2n/3+1)}\\
&= \frac{n/3}{2n/3 + 1}\\ 
&< 1.\\
\end{align*}

Finally, we consider the case $j = n/3 + 1$.  By the definition of the permutations in $\Pi_i$ and Lemma \ref{lem:ijineq}, $\epsilon_{n/3 + 1}(i) = 1$ only when $n/3 - 4 \le i \le n/3 - 1$, and so we have:

\begin{align*}
d(S_{n/3 + 1} \times S_{2n/3 -1}) = \text{ }&|S_{n/3 + 1} \times S_{2n/3 - 1}|\left( \frac{1}{|S_{n/3-4} \times S_{2n/3 + 4}|} \right.\\ 
& \left. + \frac{1}{|S_{n/3-3} \times S_{2n/3 + 3}|} + \frac{2}{|S_{n/3-2} \times S_{2n/3 + 2}|} + \frac{1}{|S_{n/3-1} \times S_{2n/3 + 1}|} \right)\\ 
= \text{ }&(n/3+1)!(2n/3-1)!\left( \frac{1}{(n/3-4)!(2n/3+4)!}  \right. \\
& \left. +  \frac{1}{(n/3-3)!(2n/3+3)!} +\frac{2}{(n/3-2)!(2n/3+2)!}\right. \\
& \left. +\frac{1}{(n/3-1)!(2n/3+1)!} \right) \\
<  \text{ } & (1/2)^5 + (1/2)^4 + 2(1/2)^3 + 1/2 \\
<  \text{ } & 1.
\end{align*}

Therefore, if $n/3 \le j < n/2$, then $d(S_j \times S_{n-j}) < 1$.
\end{proof}

\begin{thm}
\label{thm:largen}
Let $n \equiv 0 \pmod 6$ and $n \ge 30$.  The covering number of $S_n$ is $\frac{1}{2} {\binom{n}{n/2}} + \sum\limits_{i=0}^{n/3 - 1} {\binom{n}{i}}.$  Moreover, the collection $\mathcal{C}_n$ is the unique minimal cover of the elements of $S_n$ using only maximal subgroups. 
\end{thm}

\begin{proof}
By Lemma \ref{lem:Cncover}, the collection $\mathcal{C}_n$ is a cover of the elements of $S_n$.  By Lemmas \ref{lem:dprim}, \ref{lem:dimprim}, and \ref{lem:dintrans}, $d(M) < 1$ for all maximal subgroups $M$ not in $\mathcal{C}_n$.  By Lemma \ref{lem:keylemma}, $\mathcal{C}_n$ is the unique minimal cover of the collection $\Pi$ of elements of $S_n$ that uses only maximal subgroups.  The result follows. 
\end{proof}

\begin{proof}[Proof of Theorem \ref{thm:6n}]
This follows immediately from Theorems \ref{thm:18}, \ref{thm:24}, and \ref{thm:largen}. 
\end{proof}

\noindent\textsc{Acknowledgements.}  The author would like to thank Luise-Charlotte Kappe for many interesting conversations on this topic as well as feedback on earlier drafts of this manuscript as well as the referees, whose detailed comments about errors in an earlier version made this version possible.  This work was initially started when the author was employed at the University of Western Australia, and the author acknowledges the support of the Australian Research Council Discovery
Grant DP120101336 during his time there.

\bibliographystyle{plain}
\bibliography{Snbib}

\end{document}